\documentclass[oneside,english]{amsart}
\usepackage{amsthm}
\usepackage{wasysym}
\usepackage{graphicx}
\usepackage{amssymb}

\numberwithin{equation}{section} 
\numberwithin{figure}{section} 
\theoremstyle{plain}
\theoremstyle{plain}
\newtheorem{thm}{Theorem}
  \theoremstyle{plain}
  \newtheorem{prop}[thm]{Proposition}
  \theoremstyle{plain}
  \newtheorem{lem}[thm]{Lemma}

\numberwithin{thm}{section}

\usepackage{babel}

\begin{document}

\title{The Number of Extremal Components of a Rigid Measure}

\thanks{C. A. was supported in part by an NSF REU program at Indiana University.
H. B. was supported in part by a grant from the National Science Foundation.}

\author{C. Angiuli and H. Bercovici}

\address{Department of Mathematics, Indiana University, Bloomington, IN 47405}

\email{bercovic@indiana.edu, cangiuli@indiana.edu}

\maketitle
\markboth{}{} 
\begin{abstract}
The Littlewood-Richardson rule can be expressed in terms of measures,
and the fact that the Littlewood-Richardson coefficient is one amounts
to a rigidity property of some measure. We show that the number of
extremal components of such a rigid measure can be related to easily
calculated geometric data. We recover, in particular, a characterization
of those extremal measures whose (appropriately defined) duals are
extremal as well. This result is instrumental in writing explicit
solutions of Schubert intersection problems in the rigid case.
\end{abstract}

\section{Introduction\label{sec:Introduction}}

Our main object of study is a special class of measures in the plane,
which we now define. Start with three unit vectors $w_{1},w_{2},w_{3}$
in $\mathbb{R}^{2}$ such that $w_{1}+w_{2}+w_{3}=0$.\begin{center}\begin{picture}(0,0)%
\includegraphics{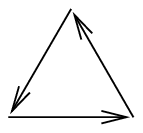}%
\end{picture}%
\setlength{\unitlength}{3947sp}%
\begingroup\makeatletter\ifx\SetFigFont\undefined%
\gdef\SetFigFont#1#2#3#4#5{%
  \reset@font\fontsize{#1}{#2pt}%
  \fontfamily{#3}\fontseries{#4}\fontshape{#5}%
  \selectfont}%
\fi\endgroup%
\begin{picture}(777,791)(1036,-470)
\put(1051, 14){\makebox(0,0)[lb]{\smash{{\SetFigFont{12}{14.4}{\familydefault}{\mddefault}{\updefault}{$w_1$}%
}}}}
\put(1756, 14){\makebox(0,0)[lb]{\smash{{\SetFigFont{12}{14.4}{\familydefault}{\mddefault}{\updefault}{$w_3$}%
}}}}
\put(1336,-406){\makebox(0,0)[lb]{\smash{{\SetFigFont{12}{14.4}{\familydefault}{\mddefault}{\updefault}{$w_2$}%
}}}}
\end{picture}%
\end{center}The
measures $\mu$ we are interested in are supported in a finite union
of lines parallel to one of these three vectors, and satisfy the following
two conditions.
\begin{enumerate}
\item On each segment which does not intersect other segments in its support,
$\mu$ is proportional to length; the constant of proportionality
is the \emph{density} of $\mu$ on that segment. The density of $\mu$
will be considered to be zero on segments outside its support.
\item For any point $A\in\mathbb{R}^{2}$, we have\begin{equation}
\delta_{1}^{+}(\mu,A)-\delta_{1}^{-}(\mu,A)=\delta_{2}^{+}(\mu,A)-\delta_{2}^{-}(\mu,A)=\delta_{3}^{+}(\mu,A)-\delta_{3}^{-}(\mu,A),\label{eq:balance-tension}\end{equation}
where $\delta_{j}^{\pm}(\mu,A)$ is the density of $\mu$ on the segment
$\{A\pm tw_{j}:t\in(0,\varepsilon)\}$ for small $\varepsilon$.
\end{enumerate}
Condition (2) is only relevant for the (finitely many) points $A$
for which at least three of the numbers $\delta_{j}^{\pm}$ are different
from zero. These are called \emph{branch points} of the measure $\mu$.
We will denote by $\mathcal{M}$ the convex cone consisting of all
measures satisfying conditions (1) and (2).

Assume now that $r$ is a positive number, and denote by $\triangle_{r}$
the (closed) triangle with vertices $0,rw_{1}$, and $r(w_{1}+w_{2})$.
The cone $\mathcal{M}_{r}\subset\mathcal{M}$ consists of those measures
$\mu$ whose branch points are contained in $\triangle_{r}$, and
whose support outside $\triangle_{r}$ consists of a finite number
of half-lines of the form $\{A+tw_{j}:t>0\}$ with $A\in\partial\triangle_{r}$
and $j\in\{1,2,3\}$. Analogously, $\mathcal{M}_{r}^{*}$ consists
of those measures $\mu$ whose branch points are contained in $\triangle_{r}$,
and whose support outside $\triangle_{r}$ is contained in a finite
number of half-lines of the form $\{A-tw_{j}:t>0\}$ with $A\in\partial\triangle_{r}$.
A point $A\in\partial\triangle_{r}$ such that $\{A\pm tw_{j}:t>0\}$
is contained in $\text{supp}(\mu)\setminus\triangle_{r}$ is called
an \emph{exit point} of $\mu$, and the corresponding density an \emph{exit
density}. The following figure shows the supports of a measure in
$\mathcal{M}_{r}$ and of a measure in $\mathcal{M}_{r}^{*}$. In
the case of $\mathcal{M}_{r}$, the boundary of $\triangle_{r}$ is
indicated by a dotted line, while for $\mathcal{M}_{r}^{*}$ the triangle
is colored light gray. The arrows indicate the half-lines in the support.

\begin{center}
\includegraphics{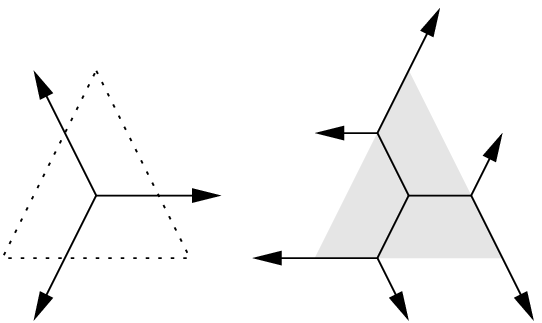}
\par\end{center}

\noindent Note that the first measure has 3 exit points, while the
second one has 6. The exit points and exit densities are always determined
by the restriction of $\mu$ to $\triangle_{r}$.

It will be useful to distinguish a subset of the exit points, called
the \emph{attachment points} of $\mu$ (or of the support of $\mu$).
An exit point of $\mu$ will be called an attachment point if either
\begin{enumerate}
\item [(a)] it is not a corner of $\triangle_{r}$, or
\item [(b)]it is a corner of $\triangle_{r}$, and the half-line through
that point, parallel to the opposite side of $\triangle_{r}$, is
in the support of $\mu$. 
\end{enumerate}
The two supports pictured above have three attachment points each.

A measure $\mu\in\mathcal{M}_{r}$ is said to be \emph{rigid} if there
is no other measure $\nu\in\mathcal{M}_{r}$ with the same exit points
and same exit densities as $\mu$. An analogous definition applies
to $\mathcal{M}_{r}^{*}$. A measure $\mu\in\mathcal{M}$, $\mu\ne0$,
is said to be \emph{extremal} if any measure $\nu\in\mathcal{M}$
satisfying $\nu\le\mu$ is of the form $c\mu$ for some constant $c$.
It was shown in \cite{bcdlt} (and the result will be reviewed below)
that rigid measures in $\mathcal{M}_{r}$ can be written in a unique
way as sums of extremal measures with distinct supports.

There is a duality which associates to each nonzero measure $\mu\in\mathcal{M}_{r}$
a measure $\mu^{*}\in\mathcal{M}_{\omega}^{*}$, where $\omega=\omega(\mu)>0$
is the weight of $\mu$, defined in Section \ref{sec:Weight-trace-etc}
below. If $\mu$ is rigid, then $\mu^{*}$ is rigid as well. We will
denote by ${\rm ext}(\mu)$ and ${\rm ext}(\mu^{*})$ the number of
extremal summands of a rigid measure $\mu$ and the corresponding
number for $\mu^{*}$. The number of attachment points of $\mu$ will
be denoted ${\rm att}(\mu)$. Our main result is as follows.
\begin{thm}
\label{thm:main}For every rigid measure $\mu\in\mathcal{M}_{r}$,
we have\[
{\rm ext}(\mu)+{\rm ext}(\mu^{*})={\rm att}(\mu)+1.\]

\end{thm}
In particular, if ${\rm ext}(\mu)={\rm ext}(\mu^{*})=1$, we deduce
that ${\rm att}(\mu)=1$, in which case the support of $\mu$ must
be in one of the positions pictured below.

\begin{center}
\includegraphics{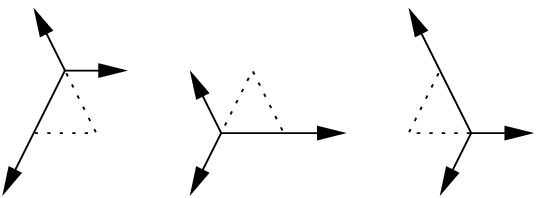}
\par\end{center}

\noindent Each of these measures has three exit points, but only one
of them is an attachment point. The attachment point is the same as
the unique branch point of the measure, and it must be deemed as \emph{two}
exit points. This is precisely \cite[Proposition 5.2]{bcdlt}.

The remainder of this paper is organized as follows. In Section \ref{sec:Weight-trace-etc}
we describe some basic properties of measures, as well as the duality
$\mu\mapsto\mu^{*}$. This material is contained, more or less explicitly,
in \cite{KT} and \cite{KTW}. Section \ref{sec:Rigidity-homology}
begins with the mechanics of decomposing rigid measures into their
extremal summands, and a linear algebraic consequence of this decomposition.
Theorem \ref{thm:secondary} shows how our main result reduces to
calculating the dimension of a certain convex set. Its proof requires
the study of measures obtained by immersing a tree (Section \ref{sec:Immersions-of-Trees})
and of the small perturbations of such immersions (Section \ref{sec:Perturbations}).
We conclude the paper with an illustration of the main result.

The role of measures in the study of the Littlewood-Richardson rule
and in the intersection theory of Grassmannians was first pointed
out in \cite{KTW}. The extremal structure of rigid measures was described
in \cite{bcdlt}, where it was shown how the associated intersections
of Schubert varieties can be written explicitly. Rigid extremal measures
have an underlying tree structure which was described in \cite{bcdlt,blt},
and which also plays a role in this paper. We should observe that
for integer values of $r$, the set $\mathcal{M}_{r}$ used in this
paper is not the same as its namesake in \cite{bcdlt} and \cite{blt}.
Indeed, in those papers the branch points are always of the form $p_{1}w_{1}+p_{2}w_{2}$
with integer $p_{1}$ and $p_{2}$. This hypothesis is natural when
dealing with intersection problems, but the arguments of \cite{bcdlt,blt}
do not depend on it in an essential manner.

\section{Weight, Trace Identity, and Dual\label{sec:Weight-trace-etc}}

Consider a measure $\mu\in\mathcal{M}_{r}$ for some $r>0$. We begin
by establishing two identities. The first one (\ref{eq:omega-def})
allows us to define the weight of $\mu$, while (\ref{eq:trace-identity})
is the trace identity which plays an important role in our arguments.

Equation (\ref{eq:balance-tension}) is equivalent to\[
\sum_{j=1}^{3}\sum_{\varepsilon=\pm}\varepsilon\delta_{j}^{\varepsilon}(\mu,A)w_{j}=0,\]
and therefore\[
\sum_{A}\sum_{j=1}^{3}\sum_{\varepsilon=\pm}\varepsilon\delta_{j}^{\varepsilon}(\mu,A)w_{j}=0,\]
where the sum is extended over all the branch points of $\mu$. Assume
that $A$ and $B$ are two branch points such that the segment $AB$
contains no other branch points of $\mu$. If $AB$ is a positive
multiple of $w_{j}$, and $\delta$ is the density of $\mu$ on $AB$,
then $\delta_{j}^{+}(\mu,A)=\delta_{j}^{-}(\mu,B)=\delta$, so that
these two terms will cancel out in the sum above. The only remaining
terms correspond therefore to the exit densities of $\mu$. Denote
by $\alpha_{1}^{(j)},\alpha_{2}^{(j)},\dots,\alpha_{k_{j}}^{(j)}$
the exit densities in the direction of $w_{j}$. We deduce that\[
\sum_{j=1}^{3}\left[\sum_{i=1}^{k_{j}}\alpha_{i}^{(j)}\right]w_{j}=0,\]
which in turn implies\begin{equation}
\sum_{i=1}^{k_{1}}\alpha_{i}^{(1)}=\sum_{i=1}^{k_{2}}\alpha_{i}^{(2)}=\sum_{i=1}^{k_{3}}\alpha_{i}^{(3)}.\label{eq:omega-def}\end{equation}
The common value of these sums is called the \emph{weight} of $\mu$,
and will be denoted $\omega(\mu)$.

There is another identity involving exit densities, which we will
call the \emph{trace identity} because of its connection with traces
of matrices. One way to deduce it is to observe that an arbitrary
measure $\mu\in\mathcal{M}$ represents the second differences of
a convex function on $\mathbb{R}^{2}$. More precisely, there exists
a (necessarily continuous) convex function $f:\mathbb{R}^{2}\to\mathbb{R}$
with the following property: for any two equilateral triangles $ABC,A'BC$
whose interiors do not intersect the support of $\mu$, we have\[
f(A)+f(A')-f(B)-f(C)=\mu(BC).\]
It suffices, of course, to require this condition for triangles whose
sides are parallel to the vectors $w_{j}$. Thus, the function $f$
is affine on each connected component in the complement of the support
of $\mu$, and the piecewise constant function $df(x+tw_{j})/dt$
jumps at the points where the line $x+tw_{j}$ intersects the support
of $\mu$ transversally, the amount of each jump being equal to the
density at the intersection point. It is easily seen that condition
(\ref{eq:balance-tension}) ensures that the various affine pieces
of $f$ fit together around each branch point of $\mu$. The function
$f$ is uniquely determined by its values at three non collinear points;
these values can be prescribed arbitrarily. We will write $\mu=\nabla^{2}f$
to indicate this relationship between $f$ and $\mu$. In the terminology
of \cite{KT} and \cite{buch-fulton-sat}, $-f$ (or its restriction
to $\triangle_{r}$) is a \emph{hive}.

Assume now that $\mu\in\mathcal{M}_{r}$ has exit densities $\{\alpha_{i}^{(j)}:1\le i\le k_{j}\}$
in the direction of $w_{j}$, $j=\{1,2,3\}$, and the corresponding
exit points are $A_{i}^{(j)}$. Denote by $X_{1}=rw_{1}$, $X_{2}=r(w_{1}+w_{2})$,
and $X_{3}=0$ the vertices of $\triangle_{r}$. The points $A_{i}^{(j)}$
are on the segment $X_{j}X_{j+1}$. We will denote by $x_{i}^{(j)}$
the distance from $A_{i}^{(j)}$ to $X_{j}$. The number $x_{i}^{(j)}\in[0,r]$
will also be called the \emph{coordinate} of $A_{i}^{(j)}$. Consider
now a convex function $f$ such that $\mu=\nabla^{2}f$. The function
$df(X_{j}+tw_{j+1})/dt$ jumps by $\alpha_{i}^{(j)}$ at $t=x_{i}^{(j)}$,
and therefore\[
f(X_{j+1})-f(X_{j})=r\beta_{j}+\sum_{i=1}^{k_{j}}\alpha_{i}^{(j)}(r-x_{i}^{(j)})=r\beta_{j}+r\omega(\mu)-\sum_{i=1}^{k_{j}}\alpha_{i}^{(j)}x_{i}^{(j)},\]
where $\beta_{j}=df(X_{j}+tw_{j+1})/dt$ for $t<0$. The function
$f$ can be chosen so that it is identically zero in the angle bounded
by $\{tw_{2}:t\ge0\}$ and $\{tw_{3}:t\ge0\}$. In this case, we have
$\beta_{3}=0$ and $\beta_{1}=\beta_{2}=-\omega(m)$, so that adding
the above identities for $j=1,2,3$ yields the \emph{trace identity}\begin{equation}
\sum_{i=1}^{k_{1}}\alpha_{i}^{(1)}x_{i}^{(1)}+\sum_{i=1}^{k_{2}}\alpha_{i}^{(2)}x_{i}^{(2)}+\sum_{i=1}^{k_{3}}\alpha_{i}^{(3)}x_{i}^{(3)}=r\omega(\mu).\label{eq:trace-identity}\end{equation}
 Let us observe that the exit point $A_{i}^{(j)}$ is an attachment
point of $\mu$ precisely when its coordinate $x_{i}^{(j)}$ is not
zero. Thus, only attachment points contribute significantly to the
sum in the left hand side.

Next we discuss the dual of a measure $\mu\in\mathcal{M}_{r}$. For
this purpose, we construct the \emph{puzzle} of $\mu|\triangle_{r}$
as follows. We translate the connected components of the complement
$\triangle_{r}\setminus{\rm supp}(\mu)$ away from each other in such
a way that the parallelogram formed by the two translates of a side
$AB$ in the support of $\mu$ has two sides which are $60^{\circ}$
clockwise from $AB$, and have length equal to the density of $\mu$
on $AB$. Condition (\ref{eq:balance-tension}) ensures that these
pieces fit together. There is a polygon, corresponding to each branch
point, which is not covered by these pieces, and which has sides equal
to the densities of the segments meeting at that point. This process
is illustrated below, where the white areas are the translated components
of $\triangle_{r}\setminus{\rm supp}(\mu)$, the connecting parallelograms
are dark gray, and the polygons corresponding to branch points are
light gray. The solid lines in the first picture represent the support
of $\mu$, and all of them are taken to have the same density for
this figure.

\begin{center}
\includegraphics{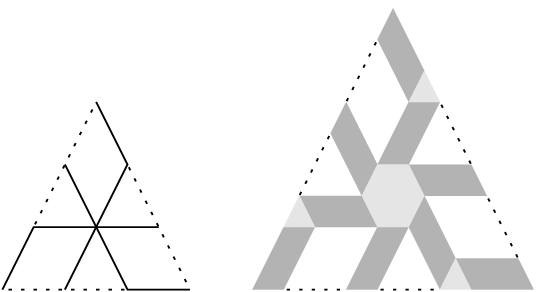}
\par\end{center}

\noindent The three kinds of pieces (white, dark gray parallelograms
 and light gray) form the puzzle of $\mu$, and the process of passing
from a measure to its puzzle is called \emph{inflation}. The puzzle
pieces cover an equilateral triangle with side $r+\omega(\mu)$ which
can, and generally will, be assumed to be precisely $\triangle_{r+\omega(\mu)}$.
We can now apply a dual process of {*}-\emph{deflation} to this puzzle
as follows. Consider a parallelogram $ABA'B'$ formed by the two translates
$AB,A'B'$ of a side in the support of $\mu$. Replace this parallelogram
with a line segment congruent to $AA'$, and assign to this segment
a density equal to the length of $AB$. Perform this operation for
all dark gray parallelograms, and discard the white pieces of the
puzzle. We obtain this way the restriction to $\triangle_{\omega(\mu)}$
of a measure $\mu^{*}\in\mathcal{M}_{\omega(\mu)}^{*}$, called the
\emph{dual} of $\mu$. For the particular measure considered above,
the support of the dual measure is pictured below. The densities are
again equal on all edges in the support.

\begin{center}
\includegraphics[scale=0.5]{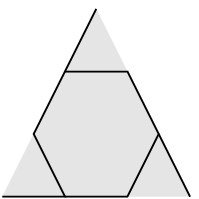}
\par\end{center}

The rigidity of a measure is equivalent to a geometric property of
its puzzle. Orient all the edges of the dark gray parallelograms in
the puzzle of $\mu$ so that they point away from the acute angles.
As shown in \cite{KTW}, $\mu$ is rigid if and only if the resulting
oriented graph has no \emph{gentle} cycle, i.e. a cycle with no sharp
turns. The measure inflated above is not rigid, as demonstrated by
the gentle cycle shown below.

\begin{center}
\includegraphics{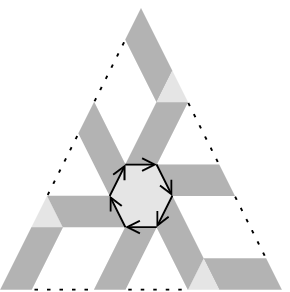}
\par\end{center}

\noindent This characterization of rigidity easily implies that the
dual of a rigid measure is also rigid. Indeed, $\mu$ and $\mu^{*}$
have the same puzzle.

\section{\label{sec:Rigidity-homology}Rigidity, Descendance, and Homology}

Consider a measure $\mu\in\mathcal{M}_{r}$ for some $r>0$, and let
$AB$ and $BC$ be two segments in the support of $\mu$ containing
no branch points in their interior. We will write $AB\to_{\mu}BC$
if one of the following two situations occurs:
\begin{enumerate}
\item $A,B,C$ are collinear, and there is a segment $BX$ with $\varangle XBC=60^{\circ}$
such that $\mu(BX)=0$;
\item $\varangle ABC=120^{\circ}$, and there is a segment $BX$ collinear
with $AB$ such that $\mu(BX)=0$.
\end{enumerate}
The relation $AB\to_{\mu}BC$ always implies that the density of $\mu$
on $BC$ is greater than or equal to the density on $AB$. More generally,
if $A_{0}A_{1},A_{1}A_{2},\dots,A_{n-1}A_{n}$ are segments in the
support of $\mu$ containing no branch points in their interior, we
write $A_{0}A_{1}\Rightarrow_{\mu}A_{n-1}A_{n}$ if $A_{i-1}A_{i}\to_{\mu}A_{i}A_{i+1}$
for $i=1,2,\dots,n$. In this case, the segment $A_{n-1}A_{n}$ is
called a \emph{descendant} of $A_{0}A_{1}$, and the sequence $A_{0}A_{1}\cdots A_{n}$
is called a \emph{descendance path}. The equivalence relation $AB\Leftrightarrow_{\mu}A'B'$
is defined by $AB\Rightarrow_{\mu}A'B'$ and $A'B'\Rightarrow_{\mu}AB$;
in order to obtain an equivalence relation, we also allow $AB\Leftrightarrow_{\mu}AB$.
A segment $AB$ in the support of $\mu$ is called a \emph{root edge}
if the relation $A'B'\Rightarrow_{\mu}AB$ implies $AB\Rightarrow_{\mu}A'B'$.
The following facts were proved in \cite[Section 3]{bcdlt} in the
case that $\mu$ is a rigid measure.
\begin{enumerate}
\item If $XY$ is in the support of $\mu$ and it is not a root edge, then
there is at least one descendance path from a root edge to $XY$.
Moreover, all descendance paths give $XY$ the same orientation.
\item For each root edge $AB$, there exists a measure $m\in\mathcal{M}_{r}$
supported by the descendants of $AB$, with density one on $AB$.
This measure $m$ is extremal, and it assigns integer densities to
all edges.
\item Let $m_{1},m_{2},\dots,m_{k}$ be the measures associated as in (2)
to a maximal family of inequivalent root edges, and let $\delta_{j}$
the density of $\mu$ on the root edges of $m_{j}$. Then we have
\[
\mu=\sum_{j=1}^{k}\delta_{j}m_{j}.\]

\end{enumerate}
Thus the number of extremal summands of $\mu$ can be determined entirely
from the geometry of the support of $\mu$. In particular, any measure
which has the same support as $\mu$ is rigid, and it can be written
as \[
\sum_{j=1}^{k}\gamma_{j}m_{j}\]
for some positive constants $\gamma_{j}$. 

Two extremal measures will be said to be \emph{inequivalent} if neither
of them is  a multiple of the other; this is equivalent to saying
that the measures have different supports. The measures $m_{1},m_{2},\dots,m_{k}$
are mutually inequivalent since the root edge of $m_{j}$ is not in
the support of $m_{i}$ for $i\ne j$.
\begin{prop}
\label{pro:linear-indep}Let $\mu\in\mathcal{M}_{r}$ be a rigid measure,
and let $A_{1},A_{2},\dots,A_{q}$ be its attachment points. Write
$\mu$ as a sum of mutually inequivalent extremal measures $\mu=\sum_{i=1}^{k}\mu_{i}$,
and denote by $\alpha_{i}^{(j)}$ the exit density of $\mu_{j}$ at
the point $A_{i}$. Then the vectors $\alpha^{(j)}=(\alpha_{i}^{(j)})_{i=1}^{q}\in\mathbb{R}^{q}$
are linearly independent. \end{prop}
\begin{proof}
Assume to the contrary that the vectors $\alpha^{(j)}$ are linearly
dependent. After a permutation of the measures, we may assume that
there exist an integer $q_{0}\in\{1,2,\dots,q\}$, and nonnegative
numbers $\beta_{j}$ such that $\beta_{1}\ne0$ and the measures $\sum_{i=1}^{q_{0}}\beta_{i}\mu_{i},\sum_{i=q_{0}+1}^{q}\beta_{i}\mu_{i}$
have the same exit densities at all attachment points. Relation (\ref{eq:omega-def})
implies that the exit densities are the same at all exit points. The
supports of these measures are contained in the support of $\mu$,
hence they are both rigid. We deduce that \[
\sum_{i=1}^{q_{0}}\beta_{i}\mu_{i}=\sum_{i=q_{0}+1}^{q}\beta_{i}\mu_{i},\]
and this is a contradiction because the measure on the right hand
side of this equation assigns zero density to some root edge of $\mu_{1}$,
unlike the left hand side.
\end{proof}
Consider now two measures $\mu\in\mathcal{M}_{r},\mu'\in\mathcal{M}_{r'}$,
and denote by $\mathcal{V}$ the collection of all vertices of white
puzzle pieces in $\triangle_{r}$ determined by the support of $\mu$.
In other words, $\mathcal{V}$ consists of the branch points and the
exit points of $\mu$, plus the corners of $\triangle_{r}$ which
are not exit points. Denote by $\mathcal{V}'$ the corresponding collection
for $\mu'$. We say that $\mu$ and $\mu'$ are \emph{homologous}
if there exists a bijection $\varphi:\mathcal{V}\to\mathcal{V}'$
such that for any two points $X,Y\in\mathcal{V}$ we have (a) the
segment $XY$ is an edge of a white piece if and only if $\varphi(X)\varphi(Y)$
is an edge of a white piece, and (b) if $XY$ is an edge of a white
piece, then $\varphi(X)\varphi(Y)$ is parallel to $XY$, and $\mu(XY)=0$
if and only if $\mu'(\varphi(X)\varphi(Y))=0$.

The characterization of rigidity in terms of gentle cycles makes it
obvious that if $\mu$ is rigid, and $\mu'$ is homologous to $\mu$,
then $\mu'$ is rigid as well.

Let us say that $\mu'$ is \emph{strictly homologous} to $\mu$ if
any edge $XY$ in the support of $\mu$ has the same density as its
homologous edge $\varphi(X)\varphi(Y)$ in the support of $\mu'$.
Our main result follows from a careful analysis of the set\[
\mathcal{H}_{\mu}=\{\mu':\mu'\text{ is strictly homologous to }\mu\}.\]
To begin with, given $\mu'\in\mathcal{H}_{\mu}$, we have $\omega(\mu')=\omega(\mu)$,
and $\mu^{\prime*}$ has precisely the same support as $\mu^{*}$.
Indeed, the lengths of the sides of the puzzle pieces of $\mu^{*}$
are precisely the densities of $\mu$. Conversely, a measure $\mu'\in\mathcal{M}_{r'}$
such that $\omega(\mu')=\omega(\mu)$ and $\text{supp}(\mu^{\prime*})={\rm supp}(\mu^{*})$
necessarily belongs to $\mathcal{H}_{\mu}$.

Now, if $\mu$ is rigid then so is $\mu^{*}$, and therefore the measures
$\nu\in\mathcal{M}_{\omega(\mu)}^{*}$ with the same support as $\mu^{*}$
are given by the general formula\[
\nu=\sum_{j=1}^{p}\gamma_{j}\nu_{j},\]
where $p={\rm ext}(\mu^{*})$, $\nu_{1},\nu_{2},\dots,\nu_{p}$ are
inequivalent extremal measures, and $\gamma_{j}>0$ for all $j$.
Thus there is a bijection between $\mathcal{H}_{\mu}$ and $\mathbb{R}_{+}^{\text{ext}(\mu^{*})}$.
Given $\gamma\in\mathbb{R}_{+}^{\text{ext}(\mu^{*})}$, we set\[
r(\gamma)=\omega\left(\sum_{j=1}^{p}\gamma_{j}\nu_{j}\right),\quad\gamma=(\gamma_{1},\gamma_{2},\dots,\gamma_{p}),\]
and denote by $m(\gamma)\in\mathcal{M}_{r(\gamma)}$ the measure satisfying\[
m(\gamma)^{*}=\sum_{j=1}^{p}\gamma_{j}\nu_{j}.\]
Denote by $\mathcal{V}_{\gamma}$ the collection of all vertices of
white puzzle in $\triangle_{r(\gamma)}$ determined by the support
of $m(\gamma)$, and let $\varphi_{\gamma}:\mathcal{V}\to\mathcal{V}_{\gamma}$
be the bijection yielding the homology of $\mu$ and $m(\gamma)$.
If $A_{1},A_{2},\dots,A_{q}$ are the attachment points of $\mu$,
then $\varphi_{\gamma}(A_{1}),\dots,\varphi_{\gamma}(A_{q})$ are
the attachment points of $m(\gamma)$. If $X_{j}(\gamma)=\varphi_{\gamma}(X_{j})$,
$j=1,2,3$ are the vertices of $\triangle_{r(\gamma)}$, we have $\varphi_{\gamma}(A_{i})=X_{j}(\gamma)+\Phi_{i}(\gamma)w_{j}$
where $\Phi_{i}(\gamma)>0$ is the coordinate of $\varphi_{\gamma}(A_{i})$.
We also set $\Phi_{0}(\gamma)=r(\gamma)$. Theorem \ref{thm:main}
follows from the following result.
\begin{thm}
\label{thm:secondary}Let $\mu$ be a rigid measure, and let the maps
$\Phi_{0},\Phi_{1},\dots,\Phi_{q}$ be as defined above, with $q=\text{{\rm att}}(\mu)$.
\begin{enumerate}
\item The map $\Phi=(\Phi_{0},\Phi_{1},\dots,\Phi_{q}):\mathbb{R}_{+}^{\text{{\rm ext}}(\mu^{*})}\to\mathbb{R}^{\text{{\rm att}}(\mu)+1}$
is linear. 
\item The map $\Phi$ is one-to-one. 
\item The range of $\Phi$ has dimension $\text{{\rm att}}(\mu)+1-\text{{\rm ext}}(\mu)$.
\end{enumerate}
\end{thm}
The linearity of $\Phi$ is easily verified. Indeed, the lengths of
the edges of white pieces in the puzzle of $m(\gamma)$ are equal
to the densities of the dual edges in the support of $\sum_{j=1}^{p}\gamma_{j}\nu_{j}$,
and these are obviously linear functions of $\gamma$. Part (2) follows
immediately from the rigidity of $\mu$. Indeed, $\Phi(\gamma)=\Phi(\gamma')$
implies that the measures $m(\gamma)$ and $m(\gamma')$ have the
same attachment points and the same exit densities, hence they must
coincide by rigidity. We conclude that $\sum_{j=1}^{p}\gamma_{j}\nu_{j}=\sum_{j=1}^{p}\gamma'_{j}\nu_{j}$,
so that $\gamma=\gamma'$ because the measures $\nu_{j}$ are linearly
independent. Assertion (3) is not as obvious, but one inequality follows
from the following result.
\begin{lem}
\label{lem:range-of-Phi}The range of $\Phi$ is contained in a subspace
$\mathbb{V}\subset\mathbb{R}^{\text{{\rm att}}(\mu)+1}$ of codimension
$\text{{\rm ext}}(\mu)$.\end{lem}
\begin{proof}
Write $\mu=\mu_{1}+\mu_{2}+\cdots+\mu_{k}$ with $k=\text{{\rm ext}}(\mu)$,
and the $\mu_{j}$ are mutually inequivalent extremal measures. Denote
by $\alpha_{i}^{(j)}$ the exit density of $\mu_{j}$ at $A_{i}$.
The trace identity (\ref{eq:trace-identity}) can be written as\[
\sum_{i=1}^{q}\alpha_{i}^{(j)}\Phi_{i}(\gamma)=\omega(\mu_{j})\Phi_{0}(\gamma),\quad\gamma\in\mathbb{R}^{p},\]
for $j=1,2,\dots,k$. According to Proposition \ref{pro:linear-indep},
these equations are linearly independent, so they define a linear
subspace of codimension $k=\text{{\rm ext}}(\mu)$.
\end{proof}

\section{\label{sec:Immersions-of-Trees}Immersions of Trees}

Extremal rigid measures have an underlying tree structure, first described
in \cite{blt}. We will consider binary trees with a finite number
of branch points (i.e., vertices of order three) and with no vertices
of order one. In other words, the leaves of the tree do not have endpoints.
Each edge between two branch points, and each leaf, will be assigned
a number in $\{1,2,3\}$ in such a way that the three numbers assigned
around each branch point are distinct. This number will be called
the \emph{type} of the edge. We will assume that each tree has at
least one branch point. Given a tree $T$, an \emph{immersion} of
$T$ is a continuous map $f:T\to\mathbb{R}^{2}$ satisfying the following
conditions:
\begin{enumerate}
\item each edge of type $j$ joining two branch points is mapped homeomorphically
onto a segment parallel to $w_{j}$, 
\item each leaf of type $j$ is mapped homeomorphically onto a half-line
parallel to $w_{j}$.
\end{enumerate}
Let us point out that the type issue was avoided in \cite{blt} by
considering only trees embedded in $\mathbb{R}^{2}$ and only orientation-preserving
immersions.

Given an immersion $f$ of a tree $T$, there is a measure $\mu_{f}\in\mathcal{M}$
supported by $f(T)$ such that the density of a segment in $f(T)$
equals the number of preimages of that segment under $f$. Measures
of the form $\mu_{f}$ are called \emph{tree measures.} Two immersions
$f,g:T\to\mathbb{R}^{2}$ yield equal tree measures provided that
$f(t)=g(t)$ for every branch point $t\in T$. In fact, even less
information suffices to determine $\mu_{f}$.
\begin{prop}
\label{pro:leaves-determine-measure}Choose a point $t_{\ell}\in\ell$
for each leaf $\ell$ of a tree $T$. Two embeddings $f,g:T\to\mathbb{R}^{2}$
yield equal measures if $f(t_{\ell})=g(t_{\ell})$ for every $\ell$.\end{prop}
\begin{proof}
As observed above, it suffices to show that the value of $f$ at each
branch point is determined by the values $f(t_{\ell})$. There must
exist leaves $\ell_{1}$ and $\ell_{2}$ which meet at a branch point
$t_{0}\in T$, and $f(t_{0})$ is then precisely the intersection
of the lines containing $f(\ell_{1})$ and $f(\ell_{2})$. These lines
are determined by $f(t_{\ell_{1}})$ and $f(t_{\ell_{2}})$ because
their directions are dictated by the types of the two leaves, and
these types are distinct. Replace now the leaves $\ell_{1}$ and $\ell_{2}$
with a single leaf $\ell$ attached at $t_{0}$, and with type different
from those of $\ell_{1}$ and $\ell_{2}$. Also define $t_{\ell}=t_{0}$.
We obtain a new tree $T'$ with one fewer leaves than $T$. Define
an immersion $f'$ of $T'$ which agrees with $f$ on the common part
of $T$ and $T'$. This operation reduces the proof of the proposition
from $T$ to $T'$, and therefore we can proceed by induction from
the trivial case of a tree with three leaves.
\end{proof}
The case of the tree with three leaves shows that the points $f(t_{\ell})$
in the above proposition must satisfy a linear equation. In the case
of tree measures in $\mathcal{M}_{r}$, this is essentially the trace
identity. The following result shows that, other than this one equation,
the points $f(t_{\ell})$ can be perturbed more or less arbitrarily.
\begin{prop}
\label{pro:epsilon-delta}Let $T$ be a tree, and $S\subset T$ a
subset such that
\begin{enumerate}
\item for each leaf $\ell$ of $T$, the intersection $\ell\cap S$ consists
of a single point $t_{\ell}$; and
\item every point $s\in S$ is of the form $t_{\ell}$ for some $\ell$.
\end{enumerate}
Fix an immersion $f:T\to\mathbb{R}^{2}$ and a point $s_{0}\in S$.
For every $\varepsilon>0$ there exists $\delta>0$ with the following
property: for any function $g_{0}:S\setminus\{s_{0}\}\to\mathbb{R}^{2}$
such that\[
|g_{0}(s)-f(s)|<\delta,\quad s\in S\setminus\{s_{0}\},\]
there exists an immersion $g$ of $T$ such that $|f(t)-g(t)|<\varepsilon$
for all $t\in T$, and $g(s)=g_{0}(s)$ for $s\in S\setminus\{s_{0}\}$.

\end{prop}
\begin{proof}
As in the preceding proposition, we proceed by induction on the number
of leaves. If $T$ has three leaves, the set $S$ contains either
one or three elements. In the case of one element, we can choose $\delta>0$
arbitrarily and define $g=f$. In the case of three elements, choose
leaves $\ell_{1},\ell_{2}$ such that $t_{\ell_{1}}\ne s_{0}\ne t_{\ell_{2}}$.
These leaves will meet at the unique branch point $t_{0}$ of $T$,
and the half-lines $f(\ell_{1})$ and $f(\ell_{2})$ meet at $f(t_{0})$.
The lines parallel to $f(\ell_{1})$ and $f(\ell_{2})$ and passing
through $g_{0}(t_{\ell_{1}})$ and $g_{0}(t_{\ell_{2}}),$ respectively,
meet at a point $A$ such that $|A-f(t_{0})|<2\delta$. The existence
of $g$ so that $g(t_{0})=A$, $g(t_{\ell_{1}})=g_{0}(t_{\ell_{1}})$,
and $g(t_{\ell_{0}})=g_{0}(t_{\ell_{0}})$ follows immediately provided
that $\delta$ is sufficiently small. Assume now that $T$ has more
than 3 leaves, and the proposition has been proved for trees with
fewer leaves. Choose two leaves $\ell_{1},\ell_{2}$ such that $t_{\ell_{1}}\ne s_{0}\ne t_{\ell_{2}}$
which intersect at $t_{0}$, and form a tree $T'$ and an immersion
$f'$ as in the preceding proof. For the tree $T'$ we choose the
set $S'=(S\setminus\{t_{\ell_{1}},t_{\ell_{2}}\})\cup\{t_{0}\}$,
and observe that $t_{\ell}=t_{0}$ for the new leaf $\ell$. Define
the map $g'_{0}:S'\backslash\{s_{0}\}\to\mathbb{R}^{2}$ by setting
$g_{0}'(s)=g_{0}(s)$ for $s\in S\setminus\{t_{\ell_{1}},t_{\ell_{2}}\}$,
while $g'(t_{0})$ is the intersection point of the lines parallel
to $f(\ell_{1})$ and $f(\ell_{2})$ and passing through $g_{0}(t_{\ell_{1}})$
and $g_{0}(t_{\ell_{2}}),$ respectively. Observe that $|g'_{0}(t_{0})-f'(t_{0})|<2\delta$.
The inductive hypothesis provides a positive number $\delta'$ corresponding
to the immersion $f'$. Choosing $\delta'\le\delta/2$ , we deduce
the existence of an immersion $g'$ of $T'$ such that $g'(s)=g_{0}(s)$
for $s\in S\setminus\{s_{0},t_{\ell_{1}},t_{\ell_{2}}\}$, $g'(t_{0})=g_{0}'(t_{0})$,
and $|g'(t)-f'(t)|<\varepsilon$ for $t\in T'$. The existence of
the required $g$ follows now easily if $\delta$ is sufficiently
small. Namely, define $g=g'$ on the common part of $T$ and $T'$,
and extend this function appropriately to the leaves $\ell_{1}$ and
$\ell_{2}$.
\end{proof}
The above proof yields $\delta\le\varepsilon/2^{b}$, where $b$ is
the number of branch points of $T$, and $b+2$ is the number of leaves.
This is in fact the best estimate for small $\varepsilon$. It is
easily seen from the preceding two proofs that the types of all the
segments of a tree are determined by the types of the leaves. Even
the type of a leaf is determined by the types of all the other leaves.

Different immersions of the same tree may yield measures which are
not homologous. An example is pictured below, where the thicker line
indicates density 2. 

\begin{center}
\includegraphics[scale=0.7]{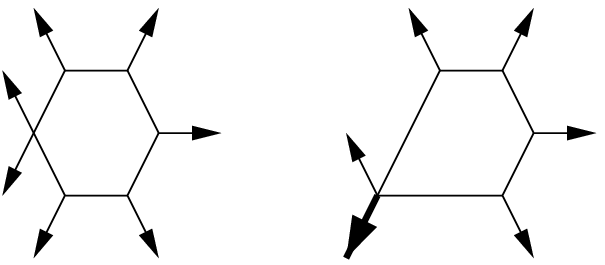}
\par\end{center}

\noindent An immersion of the tree pictured below, taking equal values
at the two points indicated by a dot, yields an immersion homologous
to the second one in the preceding figure. (The numbers indicate the
types of the leaves.)

\begin{center}
\includegraphics{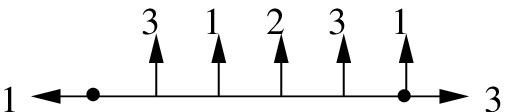}
\par\end{center}

We apply now the preceding results to a tree measure in $\mathcal{M}_{r}$.
\begin{prop}
\label{pro:perturb-one-immersion}Let $f:T\to\mathbb{R}^{2}$ be an
immersion such that $\mu_{f}\in\mathcal{M}_{r}$ for some $r>0$.
We denote by $A_{1},A_{2},\dots,A_{q}$ the attachment points of $\mu_{f}$,
by $x_{1},x_{2},\dots,x_{q}\in(0,r]$ their coordinates, and by $\alpha_{1},\alpha_{2},\dots,\alpha_{q}$
the corresponding exit densities. For every $\varepsilon>0$ there
exists $\delta>0$ with the following property: given $r'>0$ and
points $A'_{1},A'_{2},\dots,A'_{q}\in\partial\triangle_{r'}$ with
coordinates $x'_{1},\dots,x'_{q}\in(0,r']$ such that\[
\sum_{j=1}^{q}\alpha_{j}x'_{j}=\omega(\mu_{f})r'\]
and\[
|A_{j}'-A_{j}|<\delta,\quad j=1,2,\dots,q,\]
there exists an immersion $g:T\to\mathbb{R}^{2}$ such that $|f(t)-g(t)|<\varepsilon$
for all $t\in T$, $\mu_{g}$ belongs to $\mathcal{M}_{r'}$, it has
exit points $A'_{1},A'_{2},\dots,A'_{q}$ and the corresponding exit
densities are $\alpha_{1},\alpha_{2},\dots,\alpha_{q}$.\end{prop}
\begin{proof}
The case in which $\mu_{f}$ has only one exit point is trivial. Indeed,
in that case the coordinate of the exit point is equal to $r$, and
the trace identity implies that the point $A'$ has coordinate $r'$.
The map $g$ can simply be constructed as the translate $g=f+A'-A$,
and this will satisfy the requirements of the proposition if $\delta\le\varepsilon$
is sufficiently small. We will therefore assume that $\mu_{f}$ has
at least two exit points, in which case there exists $t_{0}\in T$
such that $f(t_{0})\in\triangle_{r}\setminus\partial\triangle_{r}$.
For each simple path in $T$ which starts at $t_{0}$ and ends with
one of the leaves, there exists a first point $s$ such that $f(s)\in\partial\triangle_{r}$.
We denote by $S$ the collection of all these points. We can write
$S=\bigcup_{j=1}^{q}S_{j}$ so that $f(t)=A_{j}$ for $t\in S_{j}$.
Moreover, the density $\alpha_{j}$ is precisely the cardinality of
$S_{j}$. Fix an arbitrary point $s_{0}\in S_{1}$, and let $\delta_{0}$
be provided by Proposition \ref{pro:epsilon-delta}, and choose $\delta<\delta_{0}$
such that $3\delta$ is smaller than all the segments determined on
$\partial\triangle_{r}$ by the points $A_{j}$ and the corners of
$\triangle_{r}$. For each $s\in S\setminus\{s_{0}\}$ such that $f(s)=A_{j}$,
we set $g_{0}(s)=A'_{j}$. The choice of $\delta$ implies the existence
of an immersion $g:T\to\mathbb{R}^{2}$ such that $g(s)=g_{0}(s)$
for $s\in S\backslash\{s_{0}\}$. We can also assume that the point
$A=g(s_{0})\in\partial\triangle_{r'}$, and the shortest path from
$t_{0}$ to $s_{0}$ contains no other points in $\partial\triangle_{r'}$.
The choice of $\delta$ ensures that $A$ is on the same side of $\triangle_{r'}$
as $A'_{1}$. Clearly the measure $\mu_{g}$ is in $\mathcal{M}_{r'}$,
and its attachment points are $A'_{j},j\ge2,$ with exit density $\alpha_{j}$;
$A'_{1}$ with exit density $\alpha_{1}-1$; and finally $A$ with
density $1$. To conclude the proof, we will show that in fact $A=A_{1}'.$
Denote indeed by $x$ the coordinate of $A$, and write the trace
identity for $\mu_{g}$:\[
x+(\alpha_{1}-1)x'_{1}+\sum_{j=2}^{q}\alpha_{j}x'_{j}=\omega(\mu_{g})r'.\]
Since $\omega(\mu_{g})=\omega(\mu_{f})$, this equation, combined
with the hypothesis, implies $x=x'_{1}$, and therefore $A=A_{1}'$,
as claimed.
\end{proof}

\section{\label{sec:Perturbations}Perturbations of Rigid Measures}

Fix a rigid measure $\mu$, and assume that it assigns unit density
to all of its root edges. As seen above, we can write $\mu$ as a
sum of extremal measures\[
\mu=\sum_{j=1}^{k}\mu_{j},\]
where $k=\text{{\rm ext}}(\mu)$, each $\mu_{j}$ assigns unit mass
to its root edges, and the support of $\mu_{j}$ consists of all the
descendants of some root edge $e_{j}$ of $\mu$. It was shown in
\cite{blt} that $\mu_{j}$ is of the form $\mu_{f_{j}}$ for some
immersion $f_{j}$ of a tree $T_{j}$. More precisely, choose for
each $j$ a point $P_{j}$ in the interior of $e_{j}$, and a point
$t_{j}\in T_{j}$ such that $f(t_{j})=P_{j}$. Then every simple path
starting at $t_{j}$ is mapped to a descendance path starting at $P_{j}$;
in fact, every descendance path starting at $P_{j}$ can be obtained
this way, and this is essentially how the tree $T_{j}$ is constructed.
We denote by $A_{1},A_{2},\dots,A_{q}$ the attachment points of $\mu$,
where $q=\text{{\rm att}}(\mu)$, and we let $\alpha_{i}^{(j)}$ be
the exit density of $\mu_{j}$ at the point $A_{i}$. Clearly, the
range of the map $\Phi$ considered in Theorem \ref{thm:secondary}
contains the point $(r,x_{1},x_{2},\dots,x_{q})$, where $x_{i}$
is the coordinate of the point $A_{i}$. The space $\mathbb{V}$ of
Lemma \ref{lem:range-of-Phi} consists of those triples $(r',x'_{1},x'_{2},\dots,x_{q}')\in\mathbb{R}^{\text{{\rm att}}(\mu)+1}$
satisfying the linearly independent equations\begin{equation}
\sum_{i=1}^{q}\alpha_{i}^{(j)}x'_{i}=\omega(\mu_{j})r',\quad j=1,2,\dots,k.\label{eq:bigVspace}\end{equation}
Therefore assertion (3) of Theorem \ref{thm:secondary} follows from
Lemma \ref{lem:range-of-Phi} and the following result.
\begin{prop}
\label{pro:range-open-inV}With the above notation, there exists $\delta>0$
such that any point $(r',x'_{1},x'_{2},\dots,x_{q}')\in\mathbb{R}^{\text{{\rm att}}(\mu)+1}$
satisfying $|r'-r|<\delta$, $|x'_{i}-x_{i}|<\delta$ for $i=1,2,\dots,q$,
and equations \emph{(\ref{eq:bigVspace}), }belongs to the range of
$\Phi$.\end{prop}
\begin{proof}
Denote, as before, by $\mathcal{V}$ the set consisting of all the
vertices of the polygons into which $\text{{\rm supp}}(\mu)$ divides
$\triangle_{r}$, and denote by $5\varepsilon$ the shortest distance
between two points in $\mathcal{V}$. Choose $\delta_{0}<\varepsilon$
satisfying the conclusion of Proposition \ref{pro:perturb-one-immersion}
for each of the immersions $f_{j}$, $j=1,2,\dots,k$. We will show
that our proposition is satisfied for $\delta=\delta_{0}/2$. Assume
indeed $(r',x'_{1},x'_{2},\dots,x_{q}')\in\mathbb{R}^{\text{{\rm att}}(\mu)+1}$
satisfies the hypothesis, and denote by $A'_{i}\in\partial\triangle_{r}$
the point with coordinate $x'_{i}$ such that $|A_{i}'-A_{i}|<\delta$.
Proposition \ref{pro:perturb-one-immersion} implies the existence
of immersions $g_{j}$ of $T_{j}$, $j=1,2,\dots,k$, such that the
measures $\mu_{g_{j}}$ belong to $\mathcal{M}_{r'},$ $|g_{j}(t)-f_{j}(t)|<\varepsilon$
for $t\in T_{j}$, and $\mu_{g_{j}}$ has exit density $\alpha_{i}^{(j)}$
at the point $A'_{i}$. Finally, set $\mu'=\sum_{j=1}^{k}\mu_{g_{j}}$.
The following picture illustrates a typical branch point of $\mu$,
along with its hypothetical perturbation in the support of $\mu'$.

\begin{center}\includegraphics{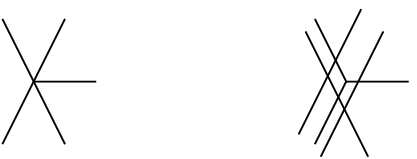}\end{center}

To conclude the proof, it will suffice to show that $\mu'$ is strictly
homologous to $\mu$. Denote by $\mathcal{V}'$ the collection of
vertices of white puzzle pieces corresponding to $\mu'$, and define
a map $\psi:\mathcal{V}'\to\mathcal{V}$ which associates to each
point $B'\in\mathcal{V}'$ the \emph{closest} point $B\in\mathcal{V}$.
Observe that the points $B'\in\mathcal{V}'$ not on the boundary of
$\triangle_{r'}$ are of two kinds. The first are of the form $g_{j}(y)$,
with $y$ a branch point of the tree $T_{j}$, and in this case we
have $\psi(B')=f_{j}(Y)$. The second kind arise as intersections
$g_{i}(e_{i})\cap g_{j}(e_{j})$, where $e_{i}$ and $e_{j}$ are
edges of $T_{i}$ and $T_{j}$ respectively, and their types are different.
In this case we have $\psi(B')=f_{i}(e_{i})\cap f_{j}(e_{j})$. In
both cases, the distance from $B'$ to $\psi(B')$ is less than $2\varepsilon$.
Hence our choice of $\varepsilon$ implies that this map is well-defined. 

Clearly we have $\psi(A'_{i})=A_{i}$ and $\psi^{-1}(A_{i})=\{A_{i}'\}$
for every $i$. More generally, for every point $B\in\mathcal{V}$,
let us denote by $\kappa(B)$ the cardinality of $\psi^{-1}(B)$,
and observe that $\kappa(B)>0$ for every $B\in\mathcal{V}$. In other
words, $\psi$ is onto. (For the above picture, we would have $\kappa=6$.)
Assume that $BC$ is a white piece edge in the support of $\mu$,
and $\mu(BC)=M$. In other words, there exist $j_{1},j_{2},\dots,j_{M}\in\{1,2,\dots,k\}$
and edges $e_{j_{i}}$ in $T_{j_{i}}$ such that $f_{j_{i}}(e_{j_{i}})$
contains $BC$ for $i=1,2,\dots,M$. Thus there are segments $e'_{j_{i}}\subset e_{j_{i}}$
such that $f_{j_{i}}(e'_{j_{i}})=BC$. The endpoints of the segments
$g_{j_{i}}(e'_{j_{i}})$ are within $2\varepsilon$ from some points
$B'_{i},C'_{i}\in\mathcal{V}'$, and in fact $\{B'_{1},B'_{2},\dots,B'_{M}\}\in\psi^{-1}(B)$
and $\{C'_{1},C'_{2},\dots,C'_{M}\}\in\psi^{-1}(C)$. The proposition
will therefore be proved if we can show that $\kappa(B)=1$ for every
$B\in\mathcal{V}$. Indeed, if that were the case we would have $\mu'(B_{1}'C_{1}')=\mu(BC)=M$
because all the edges of the trees $T_{j}$, other than $e_{j_{1}},\dots,e_{j_{M}}$,
are mapped by $g_{j}$ to segments which do not contain $B'_{1}C'_{1}$.
The map $\psi$ would thus be the bijection witnessing the strict
homology of $\mu$ and $\mu'$. 

We have already noted that $\kappa(B)=1$ if $B\in\partial\triangle_{r}$.
Select for each $j\in\{1,2,\dots,k\}$ a root edge $e_{j}$ of $\mu$
which is also a root edge for $\mu_{j}$, pick a point $P_{j}$ in
the interior of $e_{j}$, and let $t_{j}\in T_{j}$ satisfy $f(t_{j})=P_{j}$.
Given the points $A,B\in\mathcal{V}$, we will write $A\to B$ if
there is a descendance path $Y_{0}Y_{1}\cdots Y_{n}$ such that $Y_{n-1}=A,Y_{n}=P$,
and $Y_{0}=P_{j}$ for some $j\in\{1,2,\dots,k\}$. Note that we cannot
have $A\to B$ and $B\to A$. Define now sets $\mathcal{V}_{0}\subset\mathcal{V}_{1}\subset\cdots\subset\mathcal{V}$
as follows: $\mathcal{V}_{0}=\mathcal{V}\cap\partial\triangle_{r}$,
and inductively $\mathcal{V}_{n+1}$ consists of those vertices $B$
with the property that\[
\{C\in\mathcal{V}:B\to C\}\subset\mathcal{V}_{n}.\]
The properties of descendance paths for a rigid measure ensure that
we have $\mathcal{V}_{n}=\mathcal{V}$ for sufficiently large $n$.
We proceed to prove by induction on $n$ that $\kappa(B)=1$ for every
$B\in\mathcal{V}_{n}$. We already know that this is true for $n=0$,
so assume that it has been proved for all $n<N$, and let $B\in\mathcal{V}_{N}$.
Since $N>0$, $B$ is in the interior of $\triangle_{r}$, and $B$
is a branch point of $\mu$. The support of $\mu$ in the neighborhood
of $\mu$ must be in one of the following situations up to a rotation
or reflection. (This is easily deduced from the fact that the edges
in the support of $\mu$ are given a unique orientation by the relation
of descendance from any root edge. Thus, for instance, the support
of $\mu$ cannot contain six edges meeting at $B$ since this would
not allow descendance past that point.)

\noindent \begin{center}\includegraphics{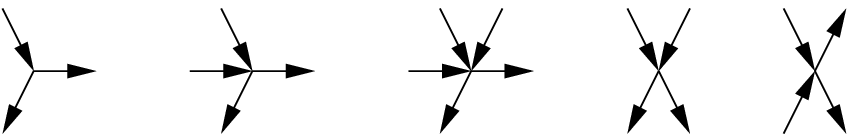}\end{center}

\noindent The arrows indicate the orientation given by descendance
from one of the points $P_{j}$. In all of these situations, there
are precisely two points $C_{1},C_{2}\in\mathcal{V}_{N-1}$ such that
$B\to C_{1}$ and $B\to C_{2}$. Moreover, $B$ is the intersection
of the lines passing through $C_{1},C_{2}$ and parallel to $w_{\ell_{1}},w_{\ell_{2}}$,
respectively, where $\ell_{1},\ell_{2}\in\{1,2,3\}$ are distinct.
The inductive hypothesis implies the existence of unique points $C'_{1}\in\psi^{-1}(C_{1})$
and $C_{2}'\in\psi^{-1}(C_{2})$. Denote by $X$ the intersection
of the lines passing through $C'_{1},C'_{2}$ and parallel to $w_{\ell_{1}},w_{\ell_{2}}$,
respectively. We will conclude the proof by showing that $\psi^{-1}(B)=\{X\}$.
Assume indeed that $B'\in\psi^{-1}(B)$. To do this, it suffices to
show that, given $j$ and a point $s_{j}\in T_{j}$ such that $f_{j}(s_{j})=B$,
the map $g_{j}$ maps some point in the neighborhood of $s_{j}$ to
$X$. There are two situations to consider.
\begin{enumerate}
\item If $s_{j}$ is a branch point for $T_{j}$, then two of the branches
are mapped by $f_{j}$ to line segments passing through $C_{1},C_{2}$
and parallel to $w_{\ell_{1}},w_{\ell_{2}}$, respectively. These
two branches are mapped by $g_{j}$ to segments through $C'_{1},C'_{2}$
and parallel to $w_{\ell_{1}},w_{\ell_{2}}$, respectively, and thus
$g(s_{j})=X$ in this case.
\item If $s_{j}$ is not a branch point for $T_{j}$, then $f_{j}$ maps
the branch containing $s_{j}$ to a line passing through $C_{i}$
and parallel to $w_{\ell_{i}}$ for $i=1$ or $i=2$. It follows that
$g_{j}$ maps this branch to a line passing through $C'_{i}$ and
parallel to $w_{\ell_{i}}$, and this line passes through $X$. It
follows that $g_{j}(s'_{j})=X$ for some $s'_{j}$ on this branch.
\end{enumerate}
The proposition follows.

\end{proof}
The rigidity assumption cannot be discarded from the hypothesis of
the preceding proposition, even for measures which are sums of rigid
tree measures. The following figure shows a sum of two tree measures
(the support of one of them in dashed lines) which is perturbed to
a measure which is not homologous to the original one.

\noindent \begin{center}
\includegraphics{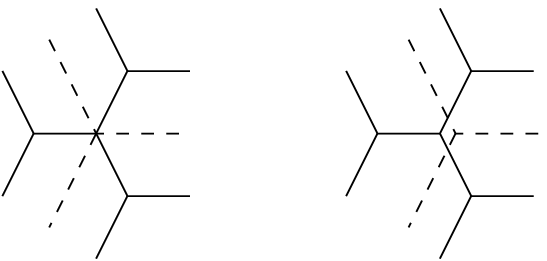}
\par\end{center}

We conclude with an illustration of our main result. The following
figure represents the support (intersected with $\triangle_{14}$)
of a rigid, extremal tree measure $\mu\in\mathcal{M}_{14}$. Clearly
we have ${\rm att}(\mu)=6$ and, if we assign unit density to its
root edges (which include, for instance, the edges of the central
pentagon) we have $\omega(\mu)=11$.

\noindent \begin{center}
\includegraphics[scale=0.5]{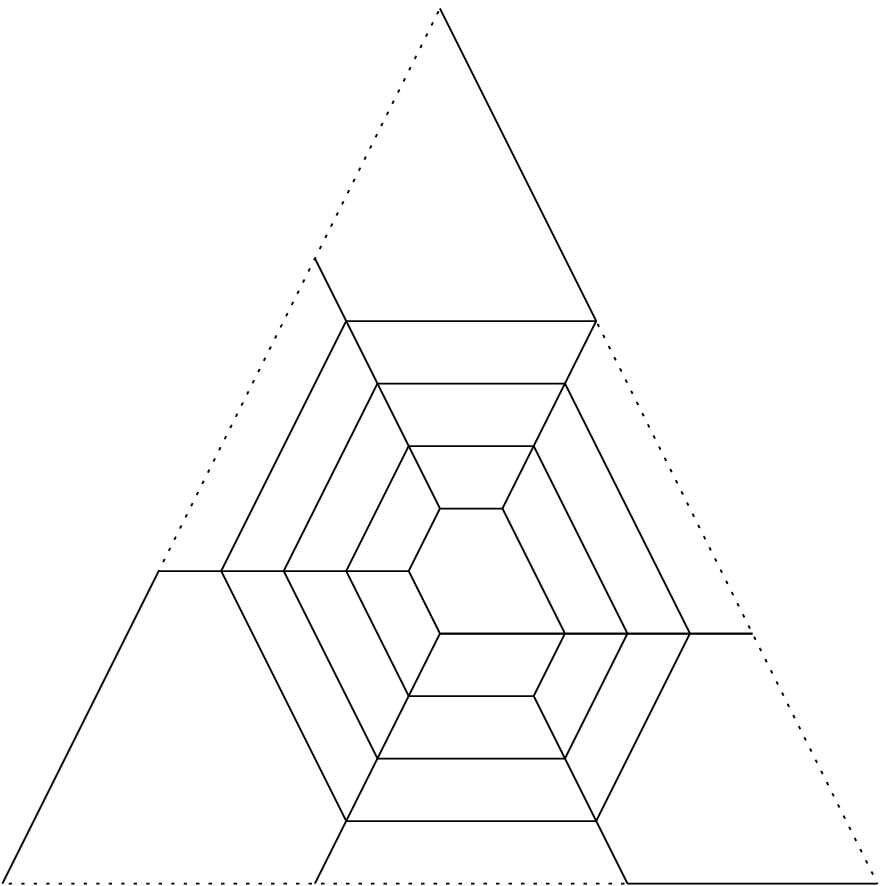}
\par\end{center}

\noindent The next figure represents ${\rm supp}(\mu^{*})\cap\triangle_{11}$.

\noindent \begin{center}
\includegraphics[scale=0.5]{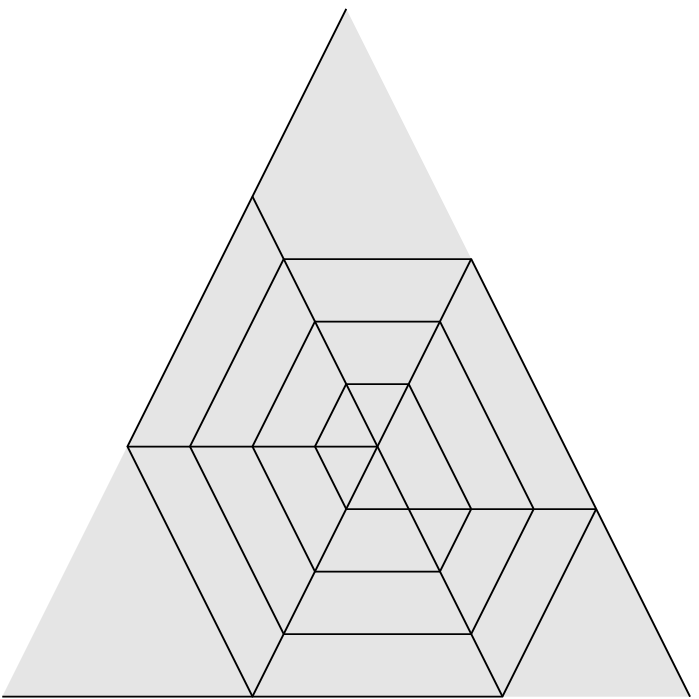}
\par\end{center}

\noindent The measure $\mu^{*}$ is the sum of six extremal measures,
as implied by Theorem \ref{thm:main}. The supports of these measures
are depicted below.

\noindent \begin{center}
\includegraphics[scale=0.5]{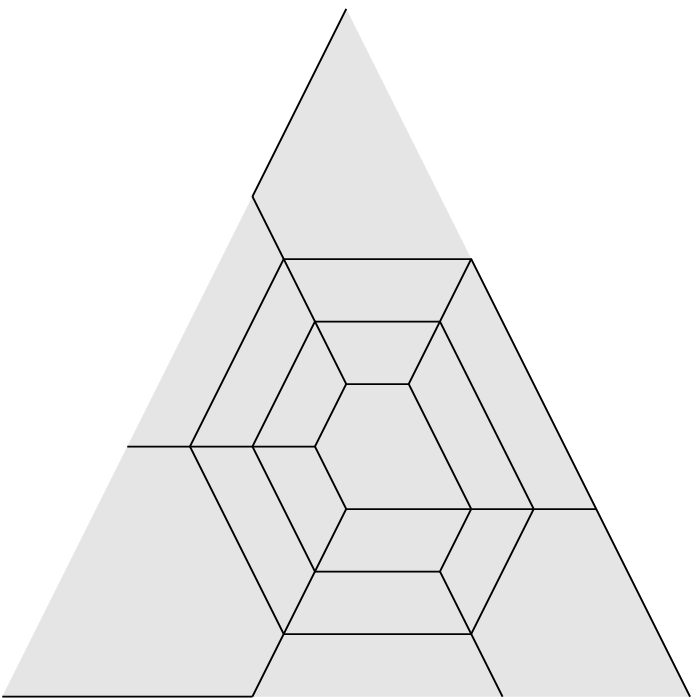}$\quad$\includegraphics[scale=0.5]{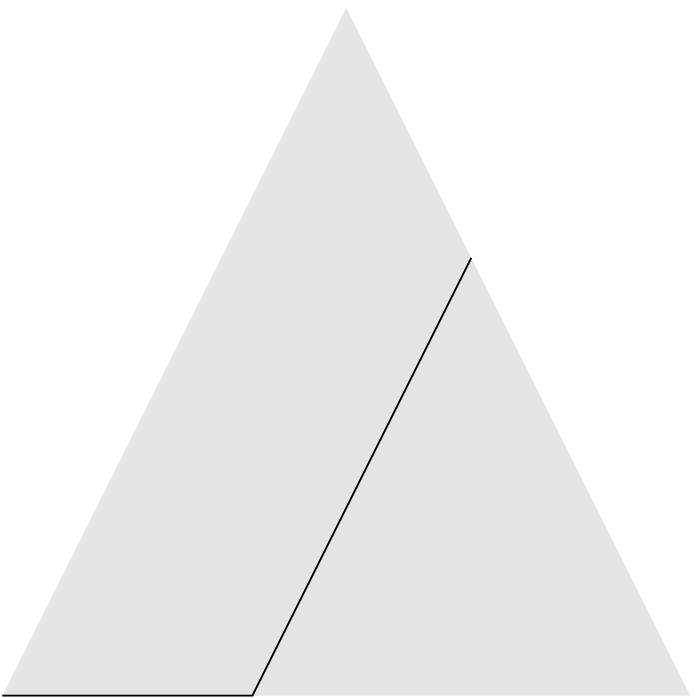}$\quad$\includegraphics[scale=0.5]{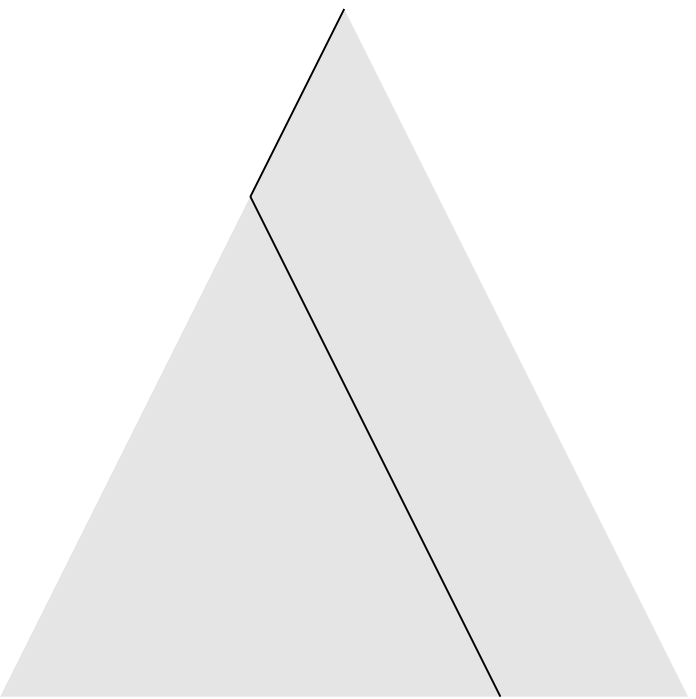}
\par\end{center}

\noindent \begin{center}
\includegraphics[scale=0.5]{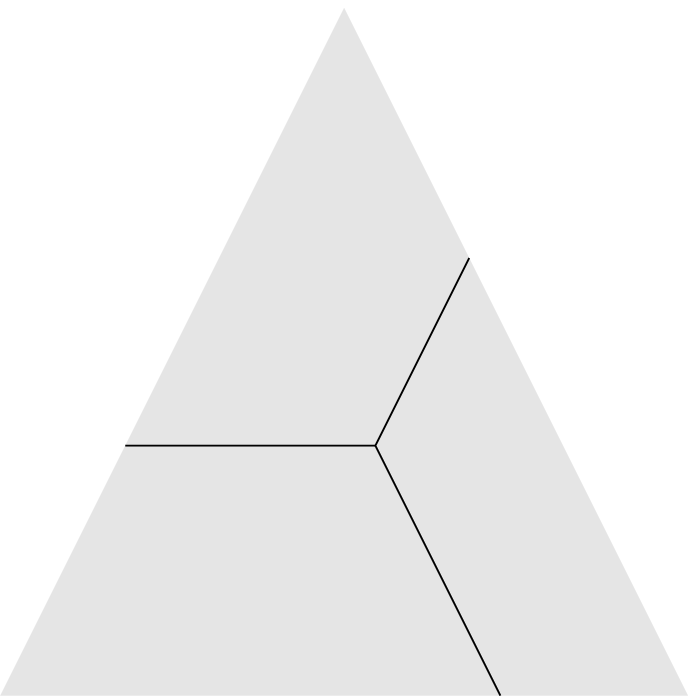}$\quad$\includegraphics[scale=0.5]{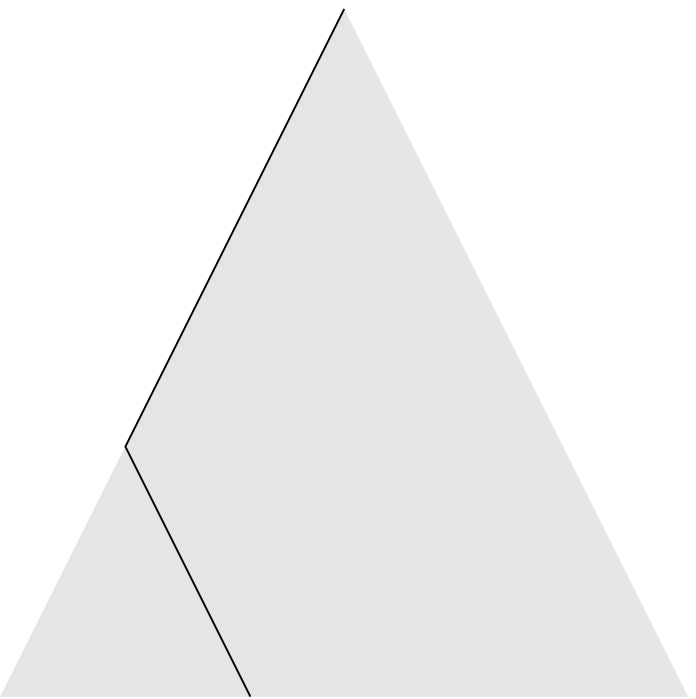}$\quad$\includegraphics[scale=0.5]{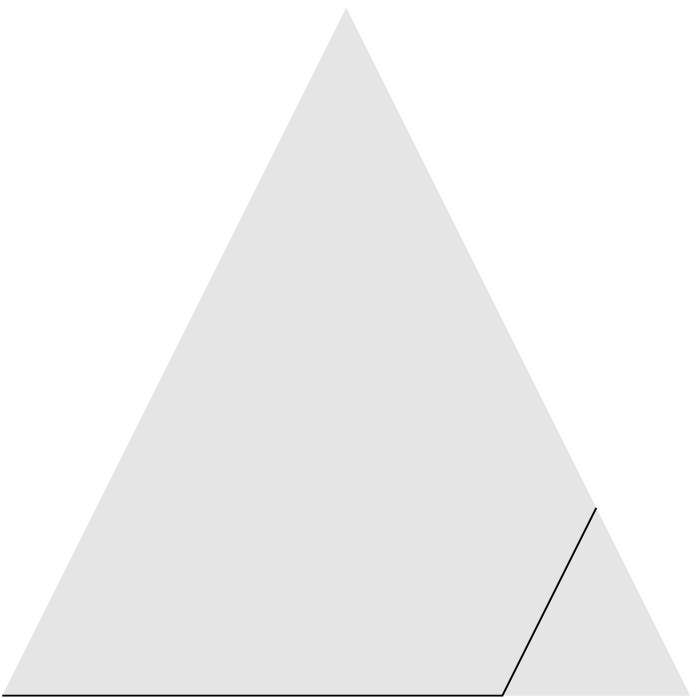}
\par\end{center}

\noindent The first of these measures has weight 9 and the remaining
five have weight 1. The reader familiar with the results of \cite{bcdlt}
will be able to verify that, among these six measures, the first measure
is the only minimal one relative to precedence.

\end{document}